\documentclass[reqno,12pt]{amsart}
\usepackage{amsfonts}
\usepackage{bbm}
\usepackage{}
\numberwithin{equation}{section}
\usepackage{color}

\usepackage{indentfirst}
\usepackage{color}
\usepackage{amssymb}
\usepackage{mathrsfs}
\usepackage{xy}
\xyoption{all}
\def\Ext{\mbox{\rm Ext}\,} \def\Hom{\mbox{\rm Hom}} \def\dim{\mbox{\rm dim}\,} 
\def\lr#1{\langle #1\rangle} \def\fin{\hfill$\square$}   
\def\Ker{\mbox{\rm Ker}\,}    \def\Coker{\mbox{\rm Coker}\,}

\def\Aut{\mbox{\rm Aut}\,}\def\A{\mathcal{A}\,} \def\H{\mathcal{H}\,}

\theoremstyle{plain}
\newtheorem{theorem}{\bf Theorem}[section]
\newtheorem{lemma}[theorem]{\bf Lemma}
\newtheorem{corollary}[theorem]{\bf Corollary}

\theoremstyle{definition}
\newtheorem{definition}[theorem]{\bf Definition}
\newtheorem{remark}[theorem]{\bf Remark}
\newtheorem{example}[theorem]{\bf Example}

\newcommand{\bt}{\begin{theorem}}
\newcommand{\et}{\end{theorem}}
\newcommand{\bl}{\begin{lemma}}
\newcommand{\el}{\end{lemma}}
\newcommand{\bd}{\begin{definition}}
\newcommand{\ed}{\end{definition}}
\newcommand{\bc}{\begin{corollary}}
\newcommand{\ec}{\end{corollary}}
\newcommand{\bp}{\begin{proof}}
\newcommand{\ep}{\end{proof}}
\newcommand{\bx}{\begin{example}}
\newcommand{\ex}{\end{example}}
\newcommand{\br}{\begin{remark}}
\newcommand{\er}{\end{remark}}
\newcommand{\be}{\begin{equation}}
\newcommand{\ee}{\end{equation}}
\newcommand{\ba}{\begin{align}}
\newcommand{\ea}{\end{align}}
\newcommand{\bn}{\begin{enumerate}}
\newcommand{\en}{\end{enumerate}}
\newcommand{\bcs}{\begin{cases}}
\newcommand{\ecs}{\end{cases}}

\makeatletter
\renewcommand{\section}{\@startsection{section}{1}{0mm}
  {-\baselineskip}{0.5\baselineskip}{\bf\leftline}}
\makeatother

\begin{document}

\title[To\"en's formula and Green's formula]{To\"en's formula and Green's formula}
\author{Haicheng Zhang}
\address{Institute of Mathematics, School of Mathematical Sciences, Nanjing Normal University,
 Nanjing 210023, P. R. China.\endgraf}
\email{zhanghc@njnu.edu.cn}

\subjclass[2010]{16G20, 17B20, 17B37.}
\keywords{Derived Hall algebra; To\"en's formula; Green's formula.}

\begin{abstract}
Let $\mathcal {A}$ be a finitary hereditary abelian category. In this note, we use the associativity of the derived Hall algebra associated to the bounded derived category of $\mathcal {A}$, whose multiplication structure constants are given by the so-called To\"en's formula, to give a simple proof of Green's formula associated to Hall algebras.
\end{abstract}

\maketitle

\section{Introduction}

The Hall algebra of a finite dimensional algebra over a finite field was introduced by Ringel \cite{R90} in 1990. Ringel \cite{R90a} proved that if $A$ is a representation finite hereditary algebra, then the twisted Hall algebra $\mathcal {H}_{v}(A)$, called the Ringel--Hall algebra, is isomorphic to the positive part of the corresponding quantized enveloping algebra. After that, Green \cite{Gr95} generalised Ringel's result to any hereditary algebra $A$ by introducing a bialgebra structure on $\mathcal {H}_{v}(A)$. It is worth mentioning that the key ingredient of the bialgebra structure of $\mathcal {H}_{v}(A)$ is the so-called Green's formula, whose original proof is given in some complicated homological methods. Then Xiao \cite{Xiao} provided the antipode for $\mathcal {H}_{v}(A)$, and obtained a realization of the whole quantized enveloping algebra by constructing the Drinfeld double of the extended Ringel-Hall algebra of $A$.

In order to give an intrinsic realization of the entire quantized enveloping algebra via Hall algebra approach,  To\"en \cite{Toen} defined a derived Hall algebra for a differential graded category satisfying some finiteness conditions, whose multiplication structure constants are given by the so-called To\"en's formula.
Later on, Xiao and Xu \cite{XiaoXu} generalised To\"en's result to any triangulated category satisfying certain finiteness conditions, which includes the bounded derived category of a finitary hereditary abelian category.

As a matter of fact, Bridgeland \cite{Br} has recently given a beautiful realization of the entire quantized enveloping algebra via the Hall algebra of $2$-cyclic complexes of projective modules over a hereditary algebra $A$, which is called Bridgeland's Hall algebra of $A$. Yanagida \cite{Yan} proved that Bridgeland's Hall algebra of a hereditary algebra is isomorphic to its Drinfeld double Hall algebra. Afterwards, Zhang \cite{ZHC} used the associativity of the Hall algebra of $A$ to give a simple proof of this result.

In order to generalise Bridgeland's construction to any hereditary abelian category which may not have enough projectives, Lu and Peng \cite{LP} introduced the modified Ringel--Hall algebra. Recently, Lin and Peng \cite{LinP} used the associativity of the modified Ringel--Hall algebra of bounded complexes of a hereditary abelian category to give a simple proof of Green's formula.

Actually, if $\A$ is a hereditary abelian category with enough projectives, it is similar to \cite{LinP} that we can also give a simple proof of Green's formula via the associativity of Bridgeland's Hall algebra of bounded complexes or $m$-cyclic $(m>2)$ complexes over projective objects, which is defined in \cite{ChenD} and further studied in \cite{ZHC2}.

In this note, inspired by the work in \cite{LinP}, we want to disclose the relation between Green's formula and To\"en's formula, which are two significant formulas in Hall algebra and derived Hall algebra, respectively. Explicitly, we will prove that To\"en's formula ``implies" Green's formula. That is, we will use the associativity of the derived Hall algebra to give a simple proof of Green's formula, without involving in modified Ringel--Hall algebras or Bridgeland's Hall algebras.

Throughout the paper, $k$ is a finite field with $q$ elements, $\A$ is a finitary (skeletally small) hereditary abelian $k$-category, and $D^b(\mathcal {A})$ is the bounded derived category of $\A$. Let $K(\mathcal{A})$ be the Grothendieck group of $\A$, and for any $M\in\A$ we denote by $\hat{M}$ the image of $M$ in $K(\mathcal{A})$. For a finite set $S$, we denote by $|S|$ its cardinality. For any object $X$ in $\A$ or $D^b(\A)$, we denote by $\Aut X$ the automorphism group of $X$, and write $a_X$ for $|\Aut X|$.

\section{Preliminaries}

In this section, we recall the definitions of Hall algebras and derived Hall algebras.
\subsection{Hall algebras}
Given $X,Y,L\in\A$,
let $$W_{XY}^L:=\{(\varphi,\psi)~|~0\rightarrow Y\xrightarrow{\varphi} L\xrightarrow{\psi} X\rightarrow 0~\text{is~exact~in}~\A\}.$$ The group $G:=\Aut{X}\times\Aut{Y}$ acts on $W_{XY}^L$ via
$$\xymatrix{&0\ar[r]&Y\ar[r]^{\varphi}\ar[d]_{g}&L\ar[r]^{\psi}\ar@{=}[d]&X\ar[r]\ar[d]^{f}&0\\
&0\ar[r]&Y\ar[r]^{\overline{\varphi}}&L\ar[r]^{\overline{\psi}}&X\ar[r]&0.}$$
That is, for any $(\varphi,\psi)\in W_{XY}^L$ and $(f,g)\in G$, $(f,g)\cdot (\varphi,\psi)=(\varphi g^{-1},f\psi)$.
We denote the set of $G$-orbits by $V_{XY}^L$. Since $\varphi$ is monic and $\psi$ epic this action is free, and we define $$g_{X Y}^{L}:=|V_{X Y}^{L}|=\frac{|W_{X Y}^{L}|}{a_Xa_Y}.$$

\begin{definition}\label{def of Hall}
The \emph{Hall algebra} $\H(\A)$ of $\A$ is the vector space over $\mathbb{C}$ with basis $\{u_{[X]}~|~X\in\A\}$ indexed by the set of isoclasses of objects in $\A$, and with multiplication defined by
\[u_{[X]} u_{[Y]} = \sum\limits_{[L]} g_{X,Y}^{L} u_{[L]}.\]
\end{definition}

It is well-known that the Hall algebra $\H(\A)$ is an associative and unital algebra.

\subsection{Derived Hall algebras}
Given $X,Y,L\in D^b(\A)$, let $$W_{XY}^L:=\{(f,g,h)~|~Y\xrightarrow{f}L\xrightarrow{g}X\xrightarrow{h}Y[1]~~\mbox{is~a~triangle~in}~~D^b(\A)\}.$$
Consider the action of $\Aut Y$ on $W_{XY}^L$ defined by
$$\xymatrix{Y\ar[r]^-f\ar[d]^-\eta&L\ar[r]^-g\ar@{=}[d]&X\ar[r]^-h\ar@{=}[d]&Y[1]\ar[d]^-{\eta[1]}\\
Y\ar[r]^-{f'}&L\ar[r]^-{g}&X\ar[r]^-{h'}&Y[1]}$$ with $\eta\in\Aut Y$.
Denote the set of orbits by $(W_{XY}^L)_Y^*$. Dually, we consider the action of $\Aut X$ on $W_{XY}^L$ and obtain the orbit set $(W_{XY}^L)_X^*$.

Since the actions above are not free, in general, $$\frac{|(W_{XY}^L)_Y^*|}{a_X}\neq\frac{|(W_{XY}^L)_X^*|}{a_Y}.$$
However, by \cite{Toen,XiaoXu} we know that
$$\frac{|(W_{XY}^L)_Y^*|}{a_X}\cdot\frac{\{L,X\}}{\{X,X\}}=\frac{|(W_{XY}^L)_X^*|}{a_Y}\cdot\frac{\{Y,L\}}{\{Y,Y\}}=:F_{X,Y}^L,$$
where $\{M,N\}:=\prod\limits_{i>0}|\Hom_{D^b(\A)}(M[i],N)|^{(-1)^i}$, for any $M,N\in D^b(\A)$. In fact, it is easy to see that
\begin{equation*}|(W_{XY}^L)_Y^*|=|\Hom_{D^b(\A)}(L,X)_{Y[1]}|~~~~\mbox{and}~~~~ |(W_{XY}^L)_X^*|=|\Hom_{D^b(\A)}(Y,L)_{X}|,\end{equation*} where $\Hom_{D^b(\A)}(M,N)_Z$ denotes the subset of $\Hom_{D^b(\A)}(M,N)$ consisting of morphisms $M\to N$ whose cone is isomorphic to $Z$. Namely, we have the following formula
$$F_{X,Y}^L=\frac{|\Hom_{D^b(\A)}(L,X)_{Y[1]}|}{a_X}\cdot\frac{\{L,X\}}{\{X,X\}}=
\frac{|\Hom_{D^b(\A)}(Y,L)_{X}|}{a_Y}\cdot\frac{\{Y,L\}}{\{Y,Y\}},$$ which is called \emph{To\"en's formula}. It is easy to see that for any $X,Y,L\in\A$ we have that $F_{X,Y}^L=g_{X,Y}^L$.

\begin{definition}
The \emph{derived Hall algebra} $\mathcal {D}\mathcal {H}(\A)$ of $\A$ is the vector space over $\mathbb{C}$ with basis $\{u_{[X]}~|~X\in D^b(\A)\}$ indexed by the set of isoclasses of objects in $D^b(\A)$, and with multiplication defined by
$$u_{[X]} u_{[Y]}=\sum\limits_{[L]}F_{X,Y}^Lu_{[L]}.$$
\end{definition}

By \cite{Toen,XiaoXu}, we know that $\mathcal {D}\mathcal {H}(\A)$ is an associative and unital algebra.

For the convenience of the calculation in the sequel, let us introduce the Euler form. Given objects $M,N \in \mathcal{A}$, define $$\lr{M,N}:=\dim_k{\Hom_{\mathcal{A}}(M,N)}-\dim_k{\Ext_{\mathcal{A}}^{1}(M,N)}.$$
It descends to give a bilinear form
$$\lr{\cdot ,\cdot }: K(\mathcal{A})\times K(\mathcal{A})\longrightarrow \mathbb{Z},$$ known as the \emph{Euler form}.

\section{Main results}
The following formulas can be viewed as generalised versions of Green's formula.
\begin{theorem}\label{main}
Given $M,N,M',N',Y\in\A$, we have the following two formulas

$(1)$~\begin{equation}
\begin{split}
&a_Ma_Na_{M'}\sum\limits_{[L],[P]}g_{MN}^Lg_{PY}^{N'}g_{M'P}^L\frac{a_P}{a_L}\\&=
\sum\limits_{[A],[A'],[B],[B'],[P']}\frac{|\Hom_{\A}(A,Y)|}{|\Hom_{\A}(A,B')|}\frac{|\Ext_{\A}^1(A,B')|}{|\Ext_{\A}^1(A,Y)|}
g_{AA'}^{M}g_{P'Y}^{B'}g_{BP'}^{N}g_{AB}^{M'}g_{A'B'}^{N'}a_{A}a_{A'}a_{B}a_{P'};
\end{split}
\end{equation}
$(2)$~\begin{equation}
\begin{split}
&a_Ma_Na_{N'}\sum\limits_{[L],[P]}g_{MN}^Lg_{PN'}^{L}g_{YP}^{M'}\frac{a_P}{a_L}\\&=
\sum\limits_{[A],[A'],[B],[B'],[P']}\frac{|\Hom_{\A}(Y,B')|}{|\Hom_{\A}(A,B')|}\frac{|\Ext_{\A}^1(A,B')|}{|\Ext_{\A}^1(Y,B')|}
g_{P'A'}^{M}g_{YP'}^{A}g_{BB'}^{N}g_{AB}^{M'}g_{A'B'}^{N'}a_{P'}a_{A'}a_{B}a_{B'}.
\end{split}
\end{equation}
\end{theorem}
\begin{corollary}{\rm\textbf{(Green's formula)}}
Given $M,N,M',N'\in\A$, we have the following formula
\begin{equation}
\begin{split}
&a_Ma_Na_{M'}a_{N'}\sum\limits_{[L]}g_{MN}^Lg_{M'N'}^L\frac{1}{a_L}\\&=
\sum\limits_{[A],[A'],[B],[B']}\frac{|\Ext_{\A}^1(A,B')|}{|\Hom_{\A}(A,B')|}
g_{AA'}^{M}g_{BB'}^{N}g_{AB}^{M'}g_{A'B'}^{N'}a_{A}a_{A'}a_{B}a_{B'}.
\end{split}
\end{equation}
\end{corollary}
\begin{proof}
Taking $Y$ in Theorem \ref{main} to be zero, we complete the proof.
\end{proof}

Before proving Theorem \ref{main}, we give two necessary lemmas.

The first lemma is well-known (for example, see \cite[Lemma 3.15]{RW} for a proof).
\begin{lemma}\label{first}
Let $f:X\to Y$ be a morphism in $\A$. Then $f$ fits into the following triangle in $D^b(\A)$
$$\xymatrix{X\ar[r]^{f}&Y\ar[r]&\Ker (f)[1]\oplus\Coker (f)\ar[r]&X[1].}$$
\end{lemma}

The following lemma is crucial to the proof of Main theorem.
\begin{lemma}\label{second}
Given $M,N,X,Y\in\A$, we have that
\begin{equation}
F_{M[1]N}^{X[1]\oplus Y}=F_{MN[-1]}^{X\oplus Y[-1]}=q^{-\lr{Y,X}}\frac{a_Xa_Y}{a_Ma_N}\sum\limits_{[L]}g_{LX}^Mg_{YL}^Na_L.
\end{equation}
\end{lemma}
\begin{proof}
This can be proved by the proof of \cite[Proposition 7.1]{Toen} and \cite[(8.8)]{Van}.
\end{proof}

\textbf{\emph{Proof of Theorem \ref{main}:}}
$(1)$~By abuse of notation, in what follows, for each $X\in D^b(\A)$, we write $u_X$ for $u_{[X]}$. On the one hand,

\begin{equation}u_{N'[1]}(u_Mu_N)=\sum\limits_{[L]}g_{MN}^Lu_{N'[1]}u_L=\sum\limits_{[L],[Y],[M']}g_{MN}^LF_{N'[1]L}^{Y[1]\oplus M'}u_{Y[1]\oplus M'},\end{equation}where we have used Lemma \ref{first} for the form of the middle term of the extension of $L$ by $N'[1]$. Since $\A$ is hereditary, we obtain that \begin{equation}u_{M'}u_{Y[1]}=F_{M'Y[1]}^{Y[1]\oplus M'}u_{Y[1]\oplus M'}.\end{equation} By \cite[Lemma 2.2]{RZ}, we have that $F_{M'Y[1]}^{Y[1]\oplus M'}=\{Y[1],M'\}\cdot|\Hom_{D^b(\A)}(Y[1],M')|=1$. So $u_{M'}u_{Y[1]}=u_{Y[1]\oplus M'}$, and thus
\begin{equation}u_{N'[1]}(u_Mu_N)=\sum\limits_{[L],[Y],[M']}g_{MN}^LF_{N'[1]L}^{Y[1]\oplus M'}u_{M'}u_{Y[1]}.\end{equation}

On the other hand,

\begin{equation}\begin{split}(u_{N'[1]}u_M)u_N&=\sum\limits_{[B'],[A]}F_{N'[1]M}^{B'[1]\oplus A}u_{A}u_{B'[1]}u_N\\&=\sum\limits_{[B'],[A],[Y],[B]}F_{N'[1]M}^{B'[1]\oplus A}F_{B'[1]N}^{Y[1]\oplus B}u_Au_Bu_{Y[1]}\\
&=\sum\limits_{[B'],[A],[Y],[B],[M']}F_{N'[1]M}^{B'[1]\oplus A}F_{B'[1]N}^{Y[1]\oplus B}g_{AB}^{M'}u_{M'}u_{Y[1].}\end{split}\end{equation}

Since $u_{N'[1]}(u_Mu_N)=(u_{N'[1]}u_M)u_N$, for each fixed $M',Y$, we have that
\begin{equation}\sum\limits_{[L]}g_{MN}^LF_{N'[1]L}^{Y[1]\oplus M'}=\sum\limits_{[B'],[A],[B]}F_{N'[1]M}^{B'[1]\oplus A}F_{B'[1]N}^{Y[1]\oplus B}g_{AB}^{M'}.\end{equation}

By Lemma \ref{second}, we obtain that
\begin{equation}\begin{split}
&\sum\limits_{[L],[P]}q^{-\lr{M',Y}}g_{MN}^L\frac{a_Ya_{M'}}{a_{N'}a_L}g_{PY}^{N'}g_{M'P}^La_P=\\&
\sum\limits_{[B'],[A],[B],[A'],[P']}q^{-\lr{A,B'}-\lr{B,Y}}\frac{a_{B'}a_{A}}{a_{N'}a_M}g_{A'B'}^{N'}g_{AA'}^Ma_{A'}
\frac{a_{Y}a_{B}}{a_{B'}a_N}g_{P'Y}^{B'}g_{BP'}^Na_{P'}g_{AB}^{M'}.
\end{split}
\end{equation}

Noting that \begin{equation*}\begin{split}&\lr{M',Y}-\lr{A,B'}-\lr{B,Y}=\lr{\hat{M}'-\hat{B},\hat{Y}}-\lr{A,B'}=\lr{\hat{A},\hat{Y}-\hat{B}'}\\&=
\frac{|\Hom_{\A}(A,Y)|}{|\Hom_{\A}(A,B')|}\frac{|\Ext_{\A}^1(A,B')|}{|\Ext_{\A}^1(A,Y)|},
\end{split}\end{equation*}we complete the proof.

$(2)$~Using $(u_{M}u_{N})u_{M'[-1]}=u_{M}(u_{N}u_{M'[-1]})$, we can prove it in a similar way as $(1)$. \fin
\section*{Acknowledgments}
This work is inspired by Ji Lin and Liangang Peng's work in \cite{LinP}.
The author is pleasantly informed that Jie Xiao and Fan Xu have also known this idea after the paper of Lin and Peng, but they do not write out.

\end{document}